\documentclass{article}

\usepackage{amsfonts}
\usepackage{verbatim}
\usepackage{amssymb}
\usepackage{amscd}
\usepackage{amsmath}
\usepackage{amsthm}
\usepackage{graphicx}
\usepackage{subfigure}
\usepackage{geometry}
\usepackage{ifpdf}
\usepackage{color}
\ifpdf
\fi

\makeatletter
\newtheorem*{rep@theorem}{\rep@title}
\newcommand{\newreptheorem}[2]{%
\newenvironment{rep#1}[1]{%
 \def\rep@title{#2 \ref{##1}}%
 \begin{rep@theorem}}%
 {\end{rep@theorem}}}
\makeatother

\newtheorem{theorem}{Theorem}[section]

\newtheorem{lemma}[theorem]{Lemma}

\newtheorem{corollary}[theorem]{Corollary}
\newtheorem{definition}[theorem]{Definition}

\renewcommand{\P}{\mathcal{P}}
\newcommand{\Q}{\mathcal{Q}}

\begin{document}

\title{On the roots of hypergraph chromatic polynomials}
\author{Sukhada Fadnavis}
\maketitle

\begin{abstract}
	Let $G = (V,E)$ be a finite, simple, connected graph with chromatic polynomial $P_G(q)$. Sokal \cite{sokal} proved that the roots of the chromatic polynomial of $G$ are bounded in absolute value by $KD$ where, $D$ is the maximum degree of the graph and $7< K < 8$ is a constant. In this paper we generalize this result to uniform hypergraphs. To prove our results we will use the theory of the bounded exponential type graph polynomials.
\end{abstract}

\section{Introduction} 

Recall that the chromatic polynomial $P_G(q)$ of a graph $G$ is defined such that $P_G(q)$ equals the number of proper colorings of $G$ with $q$ colors. We consider the following generalization of the chromatic polynomial. Let $H=(V,E)$ be a $t$--uniform hypergraph. Let $P_H(x)$ be the polynomial such that $P_H(q)$ equals the number of colorings of $H$ with $q$ colors such that there is no edge in $H$ which has all of its vertices colored with the same color. Note that in the usual case $t$ is simply $2$. We call such a coloring a \emph{proper coloring} of $H$. One can easily prove that $P_H(x)$ is indeed a polynomial which we will show in Section~\ref{GenChrom}. When we want to emphasize the role of $H$ we will also use the notation $P(H,x)$ instead of $P_H(x)$.

Sokal \cite{sokal} showed that the roots of the chromatic polynomial of $G$ are bounded above in absolute value by $KD$, where $D$ is the maximum degree of the graph and $7< K < 8$ is a constant. 

In the same spirit we show here that if  $H$ is not a graph, but a $t$-uniform hypergraph then the roots of the corresponding polynomial $P_H(q)$ are bounded in absolute value in terms of the maximum degree of $H$. Recall that the degree of a vertex of a hypergraph is simply the number of edges incident to the given vertex\footnote{Note that there are many degree concepts for hypergraphs, this is the simplest possible one.}.

\begin{theorem}\label{thm2} Let $H = (V,E)$ be a $t$ uniform hypergraph. Let $D$ be the maximum degree of $H$. Then, the roots of $P_H(x)$ are bounded above in absolute value by $8etD$.
\end{theorem}

On the way to proving this theorem, we prove an inequality, which we think is of independent interest:
\begin{theorem}\label{analogous}[Proved in Section \ref{mainproof}] Let $H = (V,E)$ be a hypergraph.  Let $N(H)$ denote the number of connected, spanning hyperforests\footnote{Note that there are multiple ways to define a connected, spanning hyperforest. The precise definition that we use is provided in Definition \ref{hypergraphdefs}.} of $H$. 
	Then, 
	\begin{equation}
		\left| \sum_{\substack{E' \subset E, \\ (V, E') connected}}  (-1)^{|E'|}  \right| \leq N(H).
	\end{equation}
\end{theorem}

This is a generalization of an analogous result for graphs due to Penrose \cite{Penrose}, where $H$ is a graph and $N(H)$ is replaced by $\tau(H)$, the number of spanning trees of $H$. Although this result has a fairly straightforward proof using the deletion-contraction recurrence relation for $\tau$, the number of spanning trees, we were not able to find an equally simple proof in the case of hypergraphs. Instead we use the theory of hypergraphic matroids developed in  \cite{Frank}  and \cite{Lorea} to prove this result in Section \ref{mainproof}. \\

This paper is organized as follows. We discuss the generalized chromatic polynomials in more detail in Section \ref{GenChrom}.  The theory of exponential type graph polynomials developed by Csikvari and Frenkel \cite{csi} is then discussed in Section \ref{ExpoType}. The proof of Theorem \ref{thm2} follows in Section \ref{Hypergraph}. The theory of hypergraphic matroids is introduced in section \ref{matroids} and finally, Theorem \ref{analogous} is proved in sub-section \ref{mainproof}.

\section{Generalized chromatic polynomials}\label{GenChrom}

We consider the following generalization of the chromatic polynomial already mentioned in the introduction. Recall that $H$ is a $t$--uniform hypergraph. Let $P_H(x)$ be the polynomial such that $P_H(q)$ equals the number of colorings of $H$ with $q$ colors such that there is no edge in $H$ with all of its  vertices colored with the same color. We call such a coloring a \emph{proper coloring} of $H$. That such a polynomial exists follows by the use of the inclusion exclusion principle:
\begin{equation}
	P_{H}(x) = \sum_{i=0}^n a_i(H) x^i =  \sum_{E' \subseteq E} x^{c(E')} (-1)^{|E'|},
\end{equation}
where $c(E')$ denotes the number of connected\footnote{A hypergraph $H'=(V',E')$ is connected if for every partition $V_1\cup V_2=V'$ there is an edge intersecting both $V_1$ and $V_2$. A connected component is a maximal induced connected subhypergraph.} components of the hypergraph $H'=(V(H),E')$.

This shows that $P_H(q)$ is indeed a polynomial\footnote{Alternatively, $P_H(q)=\sum_{j=1}^na_j(H)q(q-1)\dots (q-j+1)$ where $a_j(H)$ is the number of partitions of the vertex set into exactly $j$ sets such that no edge is contained in a set.}.\\

Then, from the above it follows that,
\begin{equation}
P_H(x) = x^n - e(H)x^{n-t+ 1} + Q_H(x)
\end{equation}
where $Q_H(x)$ is a polynomial of degree less than $n-t+1$, and $e(H)$ is the number of edges of $H$.

We note the following identity holds true for all positive integers $x,y$
\begin{equation}\label{expoeqn}
	\sum_{S \subseteq V} P(H[S], x) P(H[V(H)\setminus S]),y) = P(H, x+y),
\end{equation}
where $P(H[S], x)$ and $P(H[V(H)\setminus S]),y)$ are the generalized chromatic polynomial of the subhypergraphs induced by the subset $S$ and $V(H)\setminus S$. 
To see that  this is true, let $S$ be a subset of $V$. Then, $P(H[S],x)P(H[V(H)\setminus S], y)$ is the number of ways to color $H$ with $x+y$ (ordered) colors such that $S$ is colored with the first $x$ colors and $V(H)\setminus S$ is colored with the remaining $y$ colors.  Thus, summing over all subsets $S$ such that the vertices in $S$ are precisely the ones colored with the first $x$ colors gives equation \ref{expoeqn}.

Since this holds true for all natural numbers $x,y$, the polynomials must be equal everywhere, giving,
\begin{equation}
	\sum_{S \subseteq V} P(H[S], x) P(H[V(H)\setminus S]),y) = P(H, x+y),
\end{equation}
for all $x, y \in \mathbb{R}$.

\section{Preliminaries and lemmas: exponential type graph polynomials}\label{ExpoType}

Here we explain notation as introduced in \cite{csi}. 

\begin{definition}
	A graph polynomial is a map $f$ that maps every finite, simple graph $G = (V,E)$ to a polynomial in $\mathbb{C}[x]$. The graph polynomial is said to be monic if it is of degree $|V|$ and it has leading coefficient 1.  The graph polynomial is said to be of exponential type if $f(\emptyset, x) = 1$ and for every graph $G = (V,E)$, we have,
\begin{equation}
	\sum_{S \subseteq V} f(S, x) f(G [V \setminus S,]y) = f(G, x+y),
\end{equation}
where $f(S,x)$ is the graph polynomial of the graph induced by $G$ on $S$. 
\end{definition}

Note that this definition can be generalized to hypergraphs as follows. 

\begin{definition}
A hypergraph polynomial is a map $f$ that maps every finite, simple hypergraph $H = (V,E)$to a polynomial in $\mathbb{C}[x]$. It is said to be monic if it is of degree $|V(H)|$ and it has leading coefficient 1.  Further, it is said to be of exponential type if $f(\emptyset, x) = 1$ and for every hypergraph $H = (V,E)$, we have,
\begin{equation}
	\sum_{S \subseteq V} f(S, x) f(H [V \setminus S,]y) = f(H, x+y),
\end{equation}
where $f(S,x)$ is the hypergraph polynomial of the hypergraph induced by $H$ on $S$. 
\end{definition}

The following result due to Csikv\'ari and Frenkel \cite{csi} gives a characterization of exponential type graph polynomials via a complex function $b$ from the class of graphs with at least one vertex. This result extends to exponential-type hypergraph polynomials as the graph structure was not used in the proof of this result.

\begin{theorem}[\cite{csi}] \label{expotype}
Let $b$ be a complex-valued function on the class of graphs with at least one vertex. Define the graph polynomial $f_b$ as,
$$f_b(G,x) = \sum_{k=1}^{|V(G)|} a_k(G)x^k, $$
where, 
\begin{equation}
	a_k(G) = \sum_{\{S_1, \ldots, S_k\} \in \mathcal{P}_k}b(G[S_1]) \cdots b(G[S_k]),
\end{equation}
where the summation is over all partitions of $V(G)$ into $k$ non-empty sets.  Then, 
\begin{enumerate}
	\item For any function $b$, the graph polynomial $f_b(G,x)$ is of exponential type. 
	\item For any graph polynomial $f$ of exponential type, there exists a graph function $b$ such that $f(G,x) = f_b(G,x)$. More precisely, given that $f(G,x)$ is of exponential type, one can recover the function $b$ by setting $b(G) = a_1(G)$. 
\end{enumerate}
\end{theorem}

We need another definition before stating the results. 
\begin{definition}
	Let 
	$$f(G,x) = \sum_{i=0}^n a_i(G)x^i $$
be a monic exponential type graph polynomial. Suppose there is a function $R: \mathbb{N} \rightarrow [0,\infty)$ such that for any graph $G$ with maximum degree at most $D$, and any vertex $v \in V(G)$ and any $s \geq 1$, we have,
\begin{equation}
	\sum_{v \in S \subseteq V(G); |S| = s} |a_1(G[S])| \leq R(D)^{s-1}.
\end{equation} 
Then, we call $f$, a bounded exponential type graph polynomial. 
\end{definition}

Again we can extend this definition to hypergraph polynomials as follows:

\begin{definition}
	Let $H$ be a $t$-uniform hypergraph. 
	$$f(H,x) = \sum_{i=0}^n a_i(H)x^i $$
be a monic exponential type hypergraph polynomial. Suppose there is a function $R: \mathbb{N} \times \mathbb{N} \rightarrow [0,\infty)$ such that for any hypergraph $H$ with maximum degree at most $D$, and any vertex $v \in V(H)$ and any $s \geq 1$, we have,
\begin{equation}
	\sum_{v \in S \subseteq V(H); |S| = s} |a_1(H[S])| \leq R(D,t)^{s-1}.
\end{equation} 
Then, we call $f$, a bounded exponential type graph polynomial. 
\end{definition}

On the way to the proof of Theorem~\ref{thm2}, we will need the following result due to Csikv\'ari and Frenkel. Again this result extends to exponential type hypergraph polynomials as the graph structure was not used in the proof of this claim.

\begin{theorem}[\cite{csi}] \label{thm3} 
	Suppose $f(G,x)$ is a bounded exponential type graph polynomial, then, the absolute value of any root of $f$ is less than $c R(G)$ where  $c < 7.04$. 
\end{theorem}

Note that in the paper of Csikv\'ari and Frenkel, the above theorem was phrased in a slightly different way, but the proof of this statement is identical to the proof given there.

\section{Hypergraph chromatic polynomials are bounded exponential type}\label{Hypergraph}

In the introduction we defined the hypergraphic chromatic polynomial as:
\begin{equation}
	P_H(x) = \sum_{E' \subseteq  E} x^{c(E')} (-1)^{|E'|},
\end{equation}
where $c(E')$ denotes the number of connected components of $V$. Suppose $P_H(x) = \sum_{i=0}^n a_i(H) x^i$, then, $a_0(H) = 0$ since $c(E') \neq 0$ for any $E' \subseteq E$. Further, 
\begin{equation}
	a_1(H) = \sum_{\substack{E' \subset E, \\ (V,E') connected}}  (-1)^{|E'|}.
\end{equation} 

We have already observed that hypergraphic chromatic polynomials are exponential type. In order to show that they are bounded exponential type we need to show that exists a constant $R(H,t)$ such that for  any vertex $v \in V(H)$ and any $s \geq 1$, we have,
\begin{equation}
	\sum_{v \in S \subseteq V(H); |S| = s} |a_1(H[S])| \leq R(H,t)^{s-1}.
\end{equation} 

In this section we prove that hypergraphic chromatic polynomials are bounded exponential type. To prove the theorem, we first obtain a bound of the coefficients $a_1$. The following definitions are required in order to state these bounds: 

\begin{definition}\label{hypergraphdefs}
Let $H = (V, E)$ be a hypergraph. For $X \subseteq V$ then, $\Gamma(X)$ is defined to be the subset of edges  $e \in E$ such that $e \cap X \neq \emptyset$. A hypergraph $(V',E')$ is said to be a hypercircuit if $|V'| = |E'|$ and $ |\Gamma(X)| \geq |X|+1$ for all non-empty and proper subsets $X \subset V'$.  A hypergraph is said to be a hyperforest if it contains no hypercircuit subgraph. A spanning hyperforest is a connected hyperforest whose edges span all the vertices of $H$. 

\end{definition}
Note that in the the case of graphs, these definitions match with the standard notions of circuits and and forests. As in the case of graphs, hyperforests of hypergraphs on $n$ nodes can have at most $|V|-1$ hyperedges : 

\begin{theorem}\label{hyperforestbound} [Proved in Section \ref{mainproof}] If $H$ is a hypergraph and $\mathcal{F}$ is a hyperforest of $H$, then, $\mathcal{F}$ has at most $|V|-1$ hyperedges. 
\end{theorem}

Now, using Theorem \ref{analogous} (proved in section \ref{mainproof}), we have the following bound on $|a_1(H)|$: 

\begin{theorem} Let $H = (V,E)$ be a hypergraph.  Let $N(H)$ denote the number of connected, spanning hyperforests\footnote{Note that a connected, spanning hyperforest may not be maximal as in the case of ordinary graphs.} of $H$. 
	Then, 
	\begin{equation}
		|a_1(H)| = \left| \sum_{\substack{E' \subset E, \\ (V, E') connected}}  (-1)^{|E'|}  \right| \leq N(H).
	\end{equation}
\end{theorem}

Finally, we need the following theorem due to Sokal: 

\begin{theorem}[\cite{sokal}] \label{spanningbound} Consider a graph $G$ with maximum degree $D$. Let $\tau(G)$ denote the number spanning trees of $G$. Then, 
$$\sum_{S\subseteq V(G) \atop v\in S} \tau(G[S])\leq (eD)^{n-1}.$$
for any vertex given $v\in V(G)$.
\end{theorem}

Now, using Theorems \ref{analogous}, \ref{hyperforestbound}, \ref{spanningbound} we are ready to prove our main result: 

\begin{theorem}\label{boundedexpo}
Let $H = (V,E)$ be a $t$-uniform hypergraph with maximum degree $D$. Then, the hypergraphic graph polynomial $P_H(x)$ is of bounded exponential type with $R(H,t) = etD$. 
\end{theorem}
\begin{proof}  
To show that $P_H(x)$ is of bounded exponential type we need to show that, 
\begin{equation}
	\sum_{v \in S \subseteq V(H); |S| = s} |a_1(H[S])| \leq (etD)^{s-1}.
\end{equation} 

Let 
$$\mathcal{B} = \{ F \subseteq H : F \text{ is a connected hyperforest },  v \in V(F), |V(F)| = s \}. $$

Using the bound in Theorem \ref{analogous} it follows that,
 \begin{equation}
	\sum_{v \in S \subseteq V(H); |S| = s} |a_1(H[S])| \leq |\mathcal{B}|.
\end{equation} 

So, it remains to show that $ |\mathcal{B}| \leq (etD)^{s-1}$.
Given $H$, we consider the following graph $G_H$. The nodes of $G_H$ correspond to elements of $E$. There is an edge between two vertices $e_1, e_2$ in $G_H$ if the hyperedges $e_1, e_2$ intersect. Note that the max degree of $G_H$ is at most $tD$. Suppose $e_1, \ldots, e_j$ are the hyperedges containing vertex $v$ in $H$.
Let,
$$\mathcal{B}_i = \{ T \subseteq V(G_H) : T \text{ is a connected},  e_i \in T, |T| \leq s-1\}. $$

Let $F \in \mathcal{B}$, so $|V(F) = s$. Thus, by Theorem \ref{hyperforestbound} $F$ has at most $s-1$ edges.
Now, we have an injective map from $\mathcal{B}$ to $\cup_{i} \mathcal{B}_i$ by mapping $e_i \in E(H)$ to $e_i \in V(G_H)$. 
Finally, let,
$$T_i^j = \{ (U, E') \subseteq G_H: (U, E') \text{ connected } e_i  \in U, |U| = j, |E'| = j-1\}. $$

Then, by Theorem~\ref{spanningbound} we have $|T_i^j| \leq (etD)^{j-1} $.
Since $T \in \mathcal{B}_i$ is a connected set of vertices, we can map it to any spanning tree on $T$, thus giving an injective map from 
$$\mathcal{B}_i \rightarrow \cup_1^{s-1} T_i^j.$$ 
Putting all together, we have an injective map from 
$$\mathcal{B} \rightarrow \cup_i \cup_1^{s-1} T_i^j.$$
Thus 
$$|\mathcal{B}| \leq \sum_i \sum_{j = 1}^{s-1}  |T_i^j| \leq D  \sum_{j = 1}^{s-1}(etD)^{j-1} \leq  D^{s-1}\sum_{j = 1}^{s-1}(et)^{j-1}  \leq (etD)^{s-1}.$$
The last inequality follows since $et >2$. 

\end{proof}

The above Theorem together with Theorem \ref{thm3} complete the proof of Theorem \ref{thm2}.

\section{Hypergraphs and matroids}\label{matroids}

To prove this theorem we define the hypergraphic matroid or circuit matroid, which was first introduced in \cite{Lorea} and later reintroduced in \cite{Frank}.

\subsection{The hypergraphic matroid}

\begin{theorem}[Lorea \cite{Lorea}] \label{Lorea} Given a hypergraph $H = (V,E)$, the sub-hypergraphs which are hyperforests form the family of independent sets of a matroid on ground set $E$. 
\end{theorem}

Frank et. al. \cite{Frank} describe the rank function of the circuit-matroid. To state the theorem, we need the following definition: 
\begin{definition} For a subset $Z \subset E$, and a partition $\mathcal{P}$ of $V$, define $e_Z(\mathcal{P})$ to be the number of elements of $Z$ that have vertices in at least two parts of $\mathcal{P}$. 
\end{definition}

\begin{theorem}[Frank et. al. \cite{Frank}] \label{ranktheorem} The rank function $r_{H}$ of the circuit matroid of a hypergraph $H$ is given by the following formula: 
\begin{equation}
	r_{H}(Z) = \min_{\mathcal{P}} \{|V| - |\mathcal{P}| + e_Z(\mathcal{P}) : \mathcal{P} \text{ a partition of V}\}. 	
\end{equation}
\end{theorem}

\begin{corollary}\label{basisbound}
	Let $H = (V, E)$ be a hypergraph. Then the rank of the circuit matroid of $H$ is at most $|V|-1$. 
\end{corollary}

\begin{proof}
As seen above,
\begin{equation}
	r_{H}(Z) = \min_{\mathcal{P}} \{|V| - |\mathcal{P}| + e_Z(\mathcal{P}) : \mathcal{P} \text{ a partition of V}\}. 	
\end{equation}
	Using $\mathcal{P} = \{ V\}$, implies that the rank function is at most $|V|-1$ since $e_Z(\mathcal{P}) = 0$ in that case. 
\end{proof}

\subsection{Parition-connected hypergraphs}

Next, we define partition-connected hyper graphs, which, as we shall see, are precisely those hypergraphs that have rank $|V|-1$. Then, we will show that hypercircuits are partition-connected. Further, unions of intersecting partition-connected hypergraphs will also be shown to be partition-connected.  

\begin{definition}  Let $H = (V, E)$ be a hypergraph and let $\mathcal{P}$ be a partition of $V$. Then, let $N(\mathcal{P}) = e_{E}(\mathcal{P})$, that is, the number of hyperedges in $E$ with vertices in at least two parts of $\P$. A partition $\P$ of the vertex set is said to be a good partition of $N(\P) \geq |\P|-1$. Otherwise it is said to be a bad partition. A hypergraph is said to be partition connected if all partitions $P$ of the vertex set $V$ are good partitions. 
\end{definition}

Using theorem \ref{ranktheorem} it follows that partition connected hypergraphs have rank $|V|-1$:

\begin{corollary} A hypergraph $H = (V, E)$ is partition-connected if and only if the rank of the associated circuit-matroid $M$ is $|V|-1$. Thus, the basis elements of $M$ all have size $|V|-1$, when $H$ is partition-connected. 
\end{corollary}

\begin{proof}
	Suppose $H$ is partition-connected. The rank of the matroid equals $r_{H}(E)$. Using theorem \ref{ranktheorem} we know,
\begin{equation}\label{rankeq}
	r_{H}(E) = \min_{\mathcal{P}} \{|V| - |\mathcal{P}| + N(\mathcal{P}) : \mathcal{P} \text{ a partition of V}\}. 	
\end{equation}	
Since $H$ is partition connected, we also know that, $N(\mathcal{P}) \geq |\mathcal{P}| - 1$ for all partitions $\mathcal{P}$. Thus, 
\begin{equation}
	r_{H}(E) \leq |V|-1.	
\end{equation}	
Together with Corollary \ref{basisbound} this implies that, $r_{H}(E) = |V|-1.$

Conversely, suppose $r_H(E) = |V|-1$ for some hypergraph $H$. Then, equation \ref{rankeq} tells us that we must have $N(P) \geq |\mathcal{P}|$ for all partitions $P$ of $V$. Thus, by definition, $H$ is partition-connected. 
\end{proof} 

\begin{theorem} If hypergraph $H$ is a hypercircuit, then $H$ is partition connected. 
\end{theorem}

\begin{proof}
	Suppose $H$ is not partition connected. Hence there is a partition $\P = \{ X_1, \ldots, X_k \}$ of the vertex set such that the number of hyperedges not contained entirely in a single part is at most $|\P|-2 = k -2$. Let $N_i$ denote the number of hyperedges contained entirely in part $X_i$.  Then,
	\begin{equation}
		N_1 + \ldots + N_k + k - 2 \geq |X_1| + \ldots + |X_k| = |V|. 
	\end{equation}
Thus for some $i$ we have $N_i \geq |X_i|$. Hence, $$|\Gamma(V \setminus X_i)| \leq |E| - |N_i| \leq |V| - |X_i| =  |V \setminus X_i|.$$ This contradicts that hypothesis that $H$ is a hypercircuit. Hence $H$ must be partition connected.  
\end{proof}

\begin{lemma} \label{union} Suppose $H_1, H_2$ are partition-connected subhypergraphs of $H$ such that $V(H_1) \cap V(H_2) \neq \emptyset$. Then, $H_1 \cup H_2$ is also partition connected. 
\end{lemma}

\begin{proof}
	Consider a partition $$\P = \{V_1, \ldots, V_k, U_1, \ldots, U_h, W_1, \ldots, W_l \}$$ of $V$.  Suppose $V_i$ consist of vertices in $H_1$ but not in $H_2$, $U_i$ consist of vertices in $H_2$ but not in $H_1$ and $W_i \cap V(H_1) \cap V(H_2) \neq \emptyset$.  Let $N_1$ be the number of hyperedges in $H_1$ with vertices in at least two of the parts $\{V_1, \ldots, V_k, W_1\cap V(H_1), \ldots, W_l\cap V(H_1)\}$. Then $N_1\geq k+l-1$ by the partition connectivity of $H_1$. Let $N_2$ be the number of hyperedges in $H_2$ with vertices in at least two of the parts $\{U_1, \ldots, U_h, (W_1\cup \ldots \cup W_l)\cap V(H_2)\}$. Then $N_2\geq h$ by the  partition connectivity of $H_2$. Note that these edges intersect $V(H_2)\setminus V(H_1)$, so they are not counted in the first $N_1$ edges. Hence
the number of edges intersecting at least two parts  is at least $N_1+N_2\geq k+l+h-1$. This shows that $H_1 \cup H_2$ is partition-connected. 
\end{proof}



\subsection{The maximal bad partition}

In this section we show that the sets of vertices of the maximal partition-connected sub-hypergraphs of a hypergraph $H$ form a bad partition and it is also the `maximal bad partition' in the following sense: 

\begin{theorem} Consider a bad partition $\P = \{ V_1, \ldots, V_k\}$ that maximizes $$ f(\P) = |\P| \frac{(|V| -1)}{|V|} - N(\P).$$ Then, 
	\begin{enumerate}
		\item If $H'$ is a maximal partition-connected subhypergraph of $H$, then $V(H') \subseteq V_i$ for some $i$. 
		\item If  for a maximal partition-connected subhypergraph $H'$ of $H$, $V(H') \subseteq V_i$ then in fact, $V(H') = V_i$. 
		\item In particular, if $H_1, \ldots, H_k$ are the maximal partition-connected subhypergraphs of $H$, then it is a partition of $H$, and it is a bad partition. 
	\end{enumerate}
Hence, such a bad partition is unique and it consists of the vertex sets of all the maximal partition connected subhypergraphs of $H$. 
\end{theorem}

\begin{proof} 
 We begin by proving the first claim. If not, let $H'$ intersect $V_1, \ldots, V_m$. Then, since $H'$ is partition-connected we have, 
 $$N(\{V_1, \ldots, V_m \}) \geq m-1.$$ 
 Let $\P' = \{V_1 \cup \ldots \cup V_m, V_{m+1}, \ldots, V_k \}$. Then, 
$$N(\P') = N(\P) - N(\{ V_1, \ldots, V_m\}) \leq  k - 1 - (m-1) = k -m.$$ 
 Hence, $\P'$ is also a bad partition.  Further,
\begin{equation}
	\begin{split}
		f(\P') = |\P'| \frac{(|V| -1)}{|V|} - N(\P')  & = (|\P|-m+1)  \frac{(|V| -1)}{|V|} - N(\P) +  N(\{ V_1, \ldots, V_m\})  \\
		& \geq f(\P) - (m-1) \frac{(|V| -1)}{|V|} + m -1 \\
		& = f(\P) + \frac{(m -1)}{|V|}.  
	\end{split}
\end{equation} 
Thus, $f(\P') > f(\P)$, which contradicts the hypothesis. This proves the first claim. \\
Now suppose the second claim is false. Then, without loss of generality, $H' \subsetneq V_1$. Then, since $H'$ is a maximal partition-connected subhypergraph, $H(V_1)$ is not partition connected and hence must have a bad partition, say $\Q = \{U_1, \ldots, U_s \}$. Then, we can consider a new partition 
$$\P' = \{ U_1, \ldots, U_s, V_2, \ldots, V_k\}. $$
Note that, $|P'| = |P| + k -1$ and $N(\P') = N(\P) + N(\Q) \leq (k-2) + (s-2)  < |\P'| -1 $. Thus, $\P'$ is also a bad partition of $H$. Further, 
\begin{equation}
	\begin{split}
		f(\P') = |\P'| \frac{(|V| -1)}{|V|} - N(\P')  & = (|\P| + s -1)  \frac{(|V| -1)}{|V|} - N(\P) - N(Q)   \\
		& \geq f(\P) + (s-1) \frac{(|V| -1)}{|V|} - (s -2) \\
		& = f(\P) - \frac{(s -1)}{|V|} + 1 > f(\P). 
	\end{split}
\end{equation} 

Thus, $P'$ is a bad partition of $H$ such that $f(\P') > f(\P)$. This contradicts the original assumption and hence proves the second claim. \\
The third claim follows from Lemma~\ref{union} and the second claim.
\end{proof}

\begin{definition} A bad partition $\mathcal{P}$ that maximizes $f(\mathcal{P})$ as in the above theorem will be called the \emph{maximal bad partition} of $V$. 
\end{definition}

\subsection{Proof of Theorem~\ref{analogous}}\label{mainproof}

Finally, to complete the proof the Theorem \ref{analogous}, we first split the set of maximal spanning hyperforests $L$ of a hypergraph $H$ into equivalence classes $[L]$. Each of the equivalence classes is shown to form the independent set of a certain hypergraphic matroid. Finally we complete the proof by applying an extension of an inequality due to Penrose \cite{Penrose} to matroids.  We begin by defining the equivalence classes. 

\begin{definition} Let $\P = \{V_1, \ldots, V_k\}$ be a bad partition of the hypergraph $H'= (V,E')$.  Let $E(\P)$ be the set of hyperedges of $H$ such that they have vertices in at least two of the parts of $P$. Then, we shall call $E_{H'}(\P)$ the \emph{set of bad edges of the partition $\P$}. 
\end{definition}

\begin{definition} Let $L,K$ be spanning subhypergraphs of $H$. We say that $H \sim K$ if they have the same maximal bad partition $\P$ and if $E_L(\P) = E_K(\P)$, that is, the set of bad edges of $P$ is also the same for both $L, K$. Note that $\sim$ is an equivalence relation. Let $[L]$ denote the equivalence class of $L$ under this relation.  
\end{definition}

\begin{theorem} Let $[L]$ denote the equivalence class of a spanning subhypergraph $L$ of $H$. Let $L'$ denote the union of all the subhypergraphs in $[L]$. Let $H_1, \ldots, H_k$ be the  maximal partition-connected subhypergraphs of $L'$ and let $\mathcal{P}$ be the associated maximal bad partition. Then, any maximal hyperforest of $L'$ consists of the union of all the bad edges of the maximal bad partition $\P$ along with maximal hyperforests of $H_i$. 
\end{theorem}

\begin{proof}
	First, suppose $T$ is a maximal hyperforest of $L'$ and $T$ does not contain all the bad edges $P$. Claim: $T \cup P$ is also a hyperforest.  If not, then adding edges of $P$ must create a hypercircuit, say $A$. Suppose $V(A) \cap V(H_i) \neq \emptyset$ for some $i$. Then, since $A, H_i$ are both partition-connected, $A \cup H_i$ must also be partition connected. But $H_i$ is a maximal partition connected subhypergraph of $L'$. This is a contradiction. Hence $T \cup \P$ is a hyperforest. But this contradicts the maximality of $T$. Hence, $L'$ must contain all the edges in $P$.\\
	Next, suppose, for some $H_i$, $T$ does not contain a maximal hyperforest of $H_i$. Then, using a similar argument as above, we can add some hyper edge in $V(H_i)$ to $T$ without creating any new hypercircuits. This again contradicts the maximality of $T$. Thus, $T$ must contain maximal hyper forests of $H_i$ for all $i$.  \\
	Finally, suppose for some $H_i$, $T$ contains more than a maximal hyperforest of $H_i$, then $T$ must contain a circuit, thus contradicting that $T$ is a hyperforest. 
\end{proof} 

The following corollary immediately follows from the above theorem:

\begin{corollary}\label{cohypergraphic} Let $[L]$ denote the equivalence class of a spanning subhypergraph $L$ of $H$. Let $L'$ denote the union of all the subhypergraphs in $[L]$. Then, $[L]$ is precisely the set of all subhypergraphs of $L'$ that contain a maximal hyperforest of $H$. Hence, it follows by Theorem \ref{Lorea} that if $L = (V_L, E_L)$, then the set of independent sets of the co-hypergraphic matroid (that is, the dual matroid) $M^*_L$ of $L$ consists of the complements of the elements of $[L]$. Further let $\mathcal{B}(H, L) \subseteq [L]$ that consists of the maximal hyperforests of $H$ in $[L]$. Then the base elements of $M^*_L$ are the complements of the subgraphs in $\mathcal{B}(H, L)$.
\end{corollary}

Finally, we need the following inequality related to matroids. It is shown in \cite{bjorner} (Theorem 7.3.3 and Proposition 7.2.2) and also mentioned in \cite{sokal} that every matroid complex is shellable, hence partitionable (we will not get into the definitions here), from which immediately follows the following corollary: 

\begin{corollary}\label{matroid-ineq} Suppose $M$ is a matroid whose set of independent sets is $\mathcal{I}$ and sets of base elements is $\mathcal{B}$, and ground set $E$. For $S \in \mathcal{I}$ let $|S|$ denote the size of set $S$. Then, 

	\begin{equation}
		\left| \sum_{S \in \mathcal{I}} (-1)^{|S|} \right| \leq |\mathcal{B}|.
	\end{equation}

\end{corollary}

Putting together Corollary \ref{cohypergraphic} and Theorem \ref{matroid-ineq} gives the following inequality:

\begin{equation}
	\left| \sum_{S \in [L]} (-1)^{|S^{C}|} \right| \leq |\mathcal{B(H,L)}|.
\end{equation}

Hence, 
\begin{equation}
	\left| \sum_{S \in [L]} (-1)^{|S|} \right| \leq |\mathcal{B(H,L)}|.
\end{equation}

Now, summing over all the equivalence classes $[L]$, gives us Theorem \ref{analogous}. 

\section{Acknowledgement}
The author would like to thank P\'eter Csikv\'ari for many useful conversations and encouragement.

\end{document}